\documentclass{amsart}
\usepackage{amscd}
\usepackage[pagebackref]{hyperref}
\usepackage{amssymb}
\include{diagrams}

\newcommand{\srt}[2]{\!\sqrt[\raisebox{0.2ex}{$\scriptstyle{#1\,\,}$}]{#2}\,}

\newcommand{\Q}{{\mathbb Q}}
\newcommand{\R}{{\mathbb R}}
\newcommand{\N}{{\mathbb N}}
\newcommand{\Z}{{\mathbb Z}}
\newcommand{\OO}{\mathcal O}
\newcommand{\fp}{\mathfrak p}

\newcommand{\fa}{\mathfrak a}
\newcommand{\fb}{\mathfrak b}

\newcommand{\tH}{\widehat{\rm H}}

\newcommand{\Ram}{\operatorname{Ram}}
\newcommand{\Am}{\operatorname{Am}}
\newcommand{\pAm}{\operatorname{Am}_p}
\newcommand{\Cl}{\operatorname{Cl}}
\newcommand{\pCl}{{}_p\!\Cl}

\newcommand{\NCl}{{}_N\!\Cl}
\newcommand{\nCl}{{}_\nu\!\Cl}

\newcommand{\Gal}{\operatorname{Gal}\,}
\newcommand{\rank}{\operatorname{rank}\,}

\newcommand{\prank}{\operatorname{rank}_p\,}
\newcommand{\eps}{\varepsilon}
\newcommand{\la}{\langle}
\newcommand{\ra}{\rangle}

\newcommand{\hra}{\hookrightarrow}
\newcommand{\lra}{\longrightarrow}

\newcounter{Tc}
\newcounter{Pc}
\newcounter{Lc}

\newtheorem{thm}[Tc]{Theorem}
\newtheorem{prop}[Pc]{Proposition}

\newtheorem{cor}{Corollary}

\theoremstyle{definition}

\newtheorem*{rem}{Remark}

\title{Ideal class groups of cyclotomic number fields II}
\author{Franz Lemmermeyer}
\address{M\"orikeweg 1 \\73489 Jagstzell \\Germany}
\email{hb3@ix.urz.uni-heidelberg.de}
\subjclass{Primary 11 R 21; Secondary 11 R 29, 11 R 18}
\begin{document}

\abstract
We first study some families of maximal real subfields of cyclotomic 
fields with even class number, and then explore the implications
of large  plus class numbers of cyclotomic fields. We also discuss
capitulation of the minus part and the behaviour of $p$-class groups
in cyclic ramified $p$-extensions.
\endabstract

\maketitle

This is a continuation of \cite{L1}; parts I and II are
independent, but will be used in part III.

\section{The $2$-Class Group}

Let $h(m)$ and $h^+(m)$ denote the class number of
$\Q(\zeta_m)$ and $\Q(\zeta_m + \zeta_m^{-1})$,
respectively, and put $h^-(m) = h(m)/h^+(m)$.
In this section we will show how results on the $2$-class field
tower of quadratic number fields can be used to improve results
of {\sc Stevenhagen} \cite{Stev} on the parity of $h^+(m)$ for 
certain composite $m$ with few prime factors.

\begin{prop}\label{Ptwo}
Let $p \equiv q \equiv 1 \bmod 4$ be primes, put 
$L = \Q(\zeta_{pq})$, and let $K$ and $K^+$ be the maximal 
$2$-extensions contained in $L$ and $L^+ = L \cap \R$,
respectively.
\begin{enumerate}
\item $2 \mid h(K^+)$ if and only if $(p/q) = 1$;
\item if $(p/q) = 1$ and $(p/q)_4 = (q/p)_4$, then $2 \mid h(F)$ 
	for every subfield $F \subseteq L$ containing 
	$\Q(\sqrt{pq}\,)$;
\item if $(p/q)_4 = (q/p)_4 = +1$, then $4 \mid h(K^+)$.
\end{enumerate}
\end{prop}

\begin{proof}
By a result of {\sc R\'edei} and {\sc Reichardt} \cite{Re,RR33}, 
the quadratic number field $k = \Q(\sqrt{pq}\,)$ admits a cyclic 
quartic extension $F/k$ which is unramified outside $\infty$ 
and which is normal over $\Q$ with $\Gal(F/\Q) \simeq D_4$, 
the dihedral group of order $8$. The last property guarantees 
that $F$ is either totally real or totally complex; {\sc Scholz}
\cite{Sch} has shown that $F$ is real if and only if
$(p/q)_4 = (q/p)_4$.

Assume that $F$ is real; then, for every subfield $M$ of $K^+$ 
containing $\Q(\sqrt{pq}\,)$, the extension $FK^+/K^+$ is 
unramified everywhere and is cyclic of degree $2$ (if
$\Q(\sqrt{p},\sqrt{q}\,) \subseteq M$) or $4$ (otherwise);
by {\sc Hilbert}'s theorem 94 this implies that the class number
of $M$ is even. 

If $F$ is totally complex, we consider the field $K^+$. In this
case, $K$ and $FK^+$ are totally complex quadratic extensions 
of $K^+$ which are unramified at the finite primes.
Let $N$ be the quadratic subextension of $FK/K^+$
different from $K$ and $FK^+$. Then $N$ is totally real and
unramified at the finite primes, and we see that $K^+$ has
even class number. 

Finally, if $(p/q)_4 = (q/p)_4 = +1$, then $k$ admits a cyclic
octic extension which is unramified outside $\infty$ and
normal over $\Q$; the same proof as above shows that
$4 \mid h(K^+)$.

The fact that $K^+$ has odd class number if $(p/q) = -1$
is given as Exercise 10.4 in \cite{Wa}; here is a short proof:
since only $p$ ramifies in the $2$-extension $K_p^+/\Q$, Theorem 10.4.
in \cite{Wa} (this is a very special case of the ambiguous class
number formula) says that $K_p^+$ has odd class number.
Now $K_p^+/\Q$ is cyclic, and $\Q(\sqrt{p}\,)$ is its unique
quadratic subfield. Since $q$ is inert in $\Q(\sqrt{p}\,)$,
we conclude that $\Q$ must be the decomposition field of $q$,
i.e. $q$ is inert in $K_p^+/\Q$. Thus $q$ is the only 
ramifying prime in $K_{pq}^+/K_p^+$, and again Theorem 10.4
proves our claim.

The fact that $h_2(K^+) = 1$ if $(p/q) = -1$ also follows from 
a result of {\sc Milgram} \cite{Mil} and the class number formula.
\end{proof}

The idea behind this proof can be found in 
{\sc van der Linden}'s paper \cite{vdL}. Our next result 
strengthens a result of {\sc Cornell} and {\sc Washington} 
\cite{CW}, who showed that $h^+(m)$ is even if $m$ is 
divisible by at least four primes $\equiv 1 \bmod 4$:

\begin{prop}
Let $m$ be an integer divisible by three distinct primes 
$\equiv 1 \bmod 4$; then $2 \mid h^+(m)$. 
\end{prop}

\begin{proof}
It is sufficient to prove the claim for $m = p_1p_2p_3$, where
the $p_j \equiv 1 \bmod 4$ are pairwise distinct primes (this
follows from the fact that $h^+(m) \mid h^+(mn)$, which
is true by class field theory, since the maximal real subfield 
of $\Q(\zeta_{m})$ does not possess unramified quadratic 
extensions inside the maximal real subfield of $\Q(\zeta_{mn})$).
If two of them are quadratic residues of each other, then
the claim follows from Prop. \ref{Ptwo}. If 
$(p_1/p_2) = (p_2/p_3) = (p_3/p_1) = -1$, then there exists
an unramified quaternion extension $L$ of $\Q(\sqrt{m}\,)$,
which is normal over $\Q$ (see \cite{Lem96}). In particular, 
$L$ is either totally
real or totally complex. If it is totally real, then the
extension $LK^+/K^+$ is unramified (where $K = \Q(\zeta_m)$ 
and $K^+$ is its maximal real subfield). If $L$ is totally complex,
then $K/K^+$ and $KL/K^+$ are two different CM-extensions of $K^+$
which are unramified outside $\infty$; the quadratic subextension
of $KL/K^+$ different from $K$ and $L$ is a totally real quadratic
unramified extension of $K^+$. This proves the claim.
\end{proof}

Yet another application of this trick is

\begin{prop}
Let $p \equiv -q \equiv -q' \equiv 1 \bmod 4$ be primes such that
$(p/q) = (p/q') = 1$. Then $2 \mid h^+(pqq')$.
\end{prop}

\begin{proof}
Consider the quadratic number field $\Q(\sqrt{-pq}\,)$; 
since $(p/q) = 1$, it has class number divisible by $4$, 
and the results of {\sc R\'edei} and {\sc Reichardt} show 
that the $4$-class field of $k$ is generated by the square 
root of $\alpha_q = x+y\sqrt{p}$, where $x,y \in \Z$ satisfy 
$x^2 - py^2 = -qz^2$; the same is true with $q$ replaced by 
$q'$. Since both $\alpha_q$ and $\alpha_{q'}$ have mixed 
signature, their product is either totally positive
or totally negative. The rest of the proof is clear.
\end{proof}

\begin{rem}
Any of the primes $p \equiv 1 \bmod 4$ in the propositions
above may be replaced by $p = 8$. Note that $(q/8)_4$ is defined
by $(q/8)_4 = (-1)^{(q-1)/8}$ for all primes $q \equiv 1 \bmod 8$.
\end{rem}

\section{Morishima's Results}

In this section we will generalize a result about the 
$2$-class group of certain cyclotomic fields first proved by 
{\sc Morishima} in \cite{Mori}. There he also proved a result 
about capitulation in cyclic unramified extensions of relative 
degree $p$, which we will give in the next section, along with 
related results which will be useful in Section \ref{S9}.

\begin{thm}\label{TM1}
Let $k^+$ be a totally real number field, and let $\fp$ be
a principal prime ideal $k^+$. Assume that the
class number of $k^+$ is divisible by some integer $n$,
and let $K^+/k^+$ be a cyclic unramified extension of relative
degree $n$. Let $k$ be a totally complex quadratic extension
of $k^+$ in which $\fp$ is ramified, and put $K = kK^+$. 
Then $\Cl(K)$ contains a subgroup of type $(\Z/2\Z)^{n-1}$.
\end{thm}

\begin{proof}
We use a lower bound for the rank of the relative class group
$$\Cl_p(K/k) = \ker (N: \Cl_p(K) \lra \Cl_p(k))$$ due to 
{\sc Jehne} \cite{Je}, who showed that, for cyclic 
extensions $K/k$ of prime degree $p$, we have 
\begin{equation} \label{EJ}
  \rank\, \Cl_p(K/k) \ge \# \Ram (K/k) - \prank E_k/H - 1. 
\end{equation}
Here $\Ram(K/k)$ denotes the set of (finite and infinite) primes
of $k$ ramified in $K$, and $H = E_k \cap NK^\times$ is the
subgroup of units which are norms of elements (or equivalently,
by {\sc Hasse}'s norm theorem, which are local norms). 

Applying this to the quadratic extension $K/k^+$, we see that
$\Ram (K/K^+)$ contains $n$ primes above $\fp$, as well as
the $(K^+:\Q)$ infinite primes; moreover, $H$ contains $E^2$
(where $E = E_{K^+}$), hence $(E:H) \mid (E:E^2) = (K^+:\Q)$,
and Jehne's estimate gives $\rank \Cl_2(K/K^+) \ge n-1$.
\end{proof}

\begin{cor}\label{TM}
Let $k$ be a complex subfield of $\Q(\zeta_p)$, let 
$K^+$ be an abelian unramified extension of $k^+$ of 
degree $n$, and put $K = kK^+$. Then $\Cl_2(K)$ contains 
a subgroup of type $(\Z/2\Z)^{n-1}$.
\end{cor}

\begin{proof}
Observe that the prime ideal above $p$ in $k^+$ is principal (it
is the relative norm of $1-\zeta_p$), and apply Theorem \ref{TM1}.
\end{proof}

Although this result might help to explain why class groups 
of real subfields of cyclotomic fields with small conductor
are small, one should not regard it as a support for
{\sc Vandiver}'s conjecture that $p \nmid h^+(p)$. Of course, 
Corollary \ref{TM} predicts that $\Cl(K)$ has a subgroup of
type $(\Z/2\Z)^{p-1}$ in this case (with some $p > 10^6$,
since {\sc Vandiver}'s conjecture holds for smaller $p$), but 
there is no reason to suspect that this should be impossible for 
fields with large degree and discriminant. In fact, {\sc Cornell}
and {\sc Washington} \cite{CW} showed that there are cyclotomic 
fields $\Q(\zeta_p)$ with $h^+(p) > p$, and more recently 
{\sc Jeannin} \cite{Jean} found many quintic cyclic fields 
with large class number.

\medskip

\noindent{\bf Example.}
Let $k$ be the quartic subfield of $\Q(\zeta_{229})$; then 
$k^+ = \Q(\sqrt{229}\,)$ has class number $3$ and Hilbert 
class field $K^+ = k^+(\alpha)$, where $\alpha^3-4\alpha-1 = 0$.
Computations with {\sc Pari} \cite{pari} give 
$\Cl(K) \simeq \Z/17\Z \times (\Z/2\Z)^4$. The subgroup of order 
$17$ comes from $\Cl(k) \simeq  \Z/3\Z \times  \Z/17\Z$, while
Corollary \ref{TM} predicts that $\Cl(K)$ contains a subgroup
of type $(\Z/2\Z)^2$. 

\medskip

The next two corollaries give examples of cyclic quartic fields with
infinite class field tower:

\begin{cor}\label{C2}
Let $p \equiv 5 \bmod 8$ be a prime; if the class number $h$ of 
$k = \Q(\sqrt{p}\,)$ is  $\ge 15$, then the class field tower of 
the quartic cyclic field $K$ of conductor $p$ is infinite
(actually this holds for any complex cyclic quartic
field $K$ containing $k$).
\end{cor}

\begin{proof}
Let $F$ be the Hilbert class field of $k$; by Corollary 
\ref{TM}, the compositum $KF$ has a class group of $2$-rank 
$r \ge h - 1$; by the criterion of {\sc Golod-Shafarevic}, 
$KF$ has  infinite $2$-class field tower if 
$r \ge 2 + 2\sqrt{1 + \rank E/E^2} = 2+2 \sqrt{2h+1}$.
If $h \ge 14$, this inequality is satisfied, and our 
claim follows (note that $h$ is odd).
\end{proof}

\noindent{\bf Example.} If $p = 13693$, then $h = 15$.

\begin{cor}
Let $p \equiv q \equiv 1 \bmod 4$ be primes such that 
$pq \equiv 5 \bmod 8$; assume that the fundamental unit $\eps$
of $k = \Q(\sqrt{pq}\,)$ has positive norm, and that $h(k) \ge 6$.
Then the two cyclic complex quartic subfields of $\Q(\zeta_p)$ 
containing $k$ have infinite class field tower.
\end{cor}

\begin{proof}
Since $\eps$ has positive norm, the prime ideals above
$p$ and $q$ are principal. Thus both ideals split in $F/k$
(we use the same terminology as above), and $\Ram(KF/F)$
contains $2n$ prime ideals. This gives $\rank \Cl_2(KF) \ge 2n-1$,
and the bound of {\sc Golod-Shafarevic} shows that $KF$ has 
infinite class field tower if $h(k) \ge 5$; since $h(k)$ 
is even, we actually have $h(k) \ge 6$. 
\end{proof}

\noindent{\bf Example.} 
a) If $p = 5$ and $q = 353$, then $h = 6$. \\ 
b) According to {\sc Schoof} \cite{Sh3}, the plus class number
of $\Q(\zeta_p)$ for $p = 3547$ equals $16777$; this implies that 
$\Q(\zeta_p)$ has infinite class field tower. \\
c) {\sc Cornacchia} \cite{Co} has shown that the cyclic quintic 
extension of conductor $3931$ has $2$-class number $2^8$; this 
implies that the subfield of degree $10$ in $\Q(\zeta_{3931})$ 
has infinite $2$-class field tower.

\medskip

Techniques similar to those used in the proof of Theorem 
\ref{TM1} were used by {\sc Martinet} \cite{Mar},
{\sc Schmithals} \cite{Schm} and {\sc Schoof} \cite{Sh} to construct
quadratic number fields with infinite class field towers;
note, however, that a related construction by {\sc Matsumura}
\cite{Mat} is incorrect: the error occurs in his proof 
of Lemma 4. In fact, here is a counter example to his Theorem 1: 
take $p = 17$, $q = -23$, $\ell = 3$; his Theorem 1 predicts that
the compositum $K$ of $\Q(\sqrt{-23},\sqrt{17})$ and the cubic
field of discriminant $-23$ has an ideal class group with
subgroup $(2,2)$. However, $\Cl(K) \simeq \Z/7\Z$ by direct
computation. 

{\sc Ozaki} \cite{Oza} found an original construction
of real abelian fields with large $\ell$-class groups;
using $\ell$-adic $L$-functions and Iwasawa theory, he
proved the following result:

\begin{prop}
There exist abelian extensions $M/\Q$
whose conductor is a product of three different primes,
such that $\rank \Cl_\ell(M)$ exceeds any given integer.
\end{prop}

\begin{proof}
Fix an odd prime $\ell$; for a prime $q \equiv 1 \bmod \ell$,
let $k_q$ denote the subfield of $\Q(\zeta_q)$ of degree $\ell$.
Choose odd primes $p$, $q$ and $r$ such that 
$p \equiv q \equiv 1 \bmod \ell$, and let $n$ be the
largest odd divisor of $r-1$ such that $p^{(r-1)/n} \equiv 
q^{(r-1)/n} \equiv 1 \bmod r$. Let $K$ be the subfield of 
$\Q(\zeta_r)$ with degree $n$. Then $L = Kk_p k_q$ is a
normal extension of $K$ with $\Gal(L/K) \simeq (\ell,\ell)$, and the
primes $p$ and $q$ split completely in $K/\Q$. Let $M$ be any
of the $\ell-1$ intermediate fields of $L/K$ different from $Kk_p$
and $Kk_q$; since these fields have conductor $pqr$, all the
primes above $p$ and $q$ in $K$ (there are exactly $2n$ such primes)
must ramify in $M/K$; since $K$ is real, it does not contain
$\zeta_\ell$, hence $\rank\!_\ell \,E/H \le \rank E/E^\ell = n-1$, 
and (\ref{EJ}) shows that
$\rank \Cl_\ell(M/K) \ge 2n - (n-1) - 1 = n$.
Since $(\Cl_\ell(M):N_{L/M} \Cl_\ell(L)) = \ell$ by class field theory,
we must have $\rank \Cl_\ell(M/K) \ge n-1$. 

Since, for given $n \in \N$, there are infinitely many
primes $r \equiv 1 \bmod n$ and $p \equiv q \equiv 1 \bmod \ell r$,
our claim follows.

Incidentally, the same argument works if we replace $k_p$ by 
the field of degree $\ell$ and conductor $\ell^2$. 
\end{proof}

\section{Capitulation of Ideal Classes}

We want to study the following situation: let $L/F$ be an
abelian extension with Galois group $G = \Gal(L/F)
\simeq \Delta \times \Gamma$, where $\Delta$ and $\Gamma$
are cyclic groups of coprime order. Let $k$ and $K$
denote the fixed fields of $\Gamma$ and $\Delta$, 
respectively; then we can identify $\Delta = \Gal(k/F) \simeq
\Gal(L/K)$ and $\Gamma = \Gal(L/k) \simeq \Gal(K/F)$ (see Figure
\ref{F1} for the Hasse diagrams).

\begin{figure}[h]
\caption{}\label{F1}
\begin{diagram}[width=1.5em,height=1.5em]
  &     &  & L  &   & & &  &     &  & 1  &   & \\
  &  & \ldLine(3,3)  & & \rdLine(2,2) & & &   &  
		& \ldLine(3,3) & & \rdLine(2,2)  &  \\
  & &   & &  & K &   &   & &   & &  & \Delta \\
 k  &   & & &  \ruLine(3,3)  & & &  \Gamma   &   & & &  \ruLine(3,3)  &  \\
   &  \luLine(2,2) &   & &  & &   &   &  \luLine(2,2) &   & &   &\\
   	&    & F &   & 	& & & &    & G &   & 	&\\
\end{diagram}
\end{figure}

Let $M$ be a $G$-module of order coprime to $\# \Delta$
(e.g. $M = \Cl_p(L)$, $\Cl_p(L/k)$, or $\kappa = \kappa_{L/k}$
for primes $p \nmid \# \Delta$);  we can decompose $M$
using the idempotents $e_\phi$ of the group ring $\Z[\Delta]$
as $M = \bigoplus M(\phi)$ with $M(\phi) = e_\phi(M)$.
Now we use (cf. {\sc Schoof} \cite{Sh2})

\begin{prop}\label{Sf}
In this notation we have 
$\tH^q(\Gamma,M)(\phi) \simeq \tH^q(\Gamma,M(\phi))$. 
\end{prop}

\begin{proof}
Let $\phi \ne \chi$ be different characters of $\Delta$, 
and consider the submodule $\tH^q(\Gamma,M(\phi))$; then 
$\tH^q(\Gamma,M(\phi))(\chi) = 0$, since $e_\chi$ kills 
the image of every $x \in \tH^q(\Gamma,M(\phi))$. Thus the 
injection $\imath: \tH^q(\Gamma,M)(\chi) \hra \tH^q(\Gamma,M)$
actually lands in  $\tH^q(\Gamma,M(\chi))$, and we have an 
injection $\imath: \tH^q(\Gamma,M)(\chi) \hra \tH^q(\Gamma,M(\chi))$. 
Summing over all inequivalent $\chi$ we get $ \tH^q(\Gamma,M)$ on 
both sides, hence $\imath$ must be an isomorphism.
\end{proof}

Since $L/k$ is cyclic, we have an injection 
$\kappa \hra \tH^{-1}(\Gamma,E_L)$ (see {\sc Iwasawa} \cite{Iwa}); 
here $\tH^q$ denotes {\sc Tate}'s cohomology groups. Taking 
the $\phi$-parts of this injection we find 
$\kappa(\phi) \hra \tH^{-1}(\Gamma,E_L)(\phi)$. 
Now Prop. \ref{Sf} shows that 
$\kappa(\phi) \hra \tH^{-1}(\Gamma,E_L(\phi))$.

As a special case, let $L$ and $k$ be CM-fields with
maximal real subfields $K$ and $F$, respectively (in
particular, $\Delta = \{1, J\}$, where $J$ denotes
complex conjugation). Then the minus part of $\tH^{-1}(\Gamma,E_L)$
is $\tH^{-1}(\Gamma,E_L^-) = \tH^{-1}(\Gamma,W_L) \simeq 
{}_NW_L/W_L^{1-\sigma}$, where ${}_NW_L = \{\zeta \in W_L:
N_{L/k} \zeta = 1\}$. We have shown (compare {\sc Jaulent}
\cite{Jau} and {\sc Kida} \cite{K}):

\begin{prop}
Let $L/k$ be a cyclic extension of CM-fields of odd prime degree $p$.
Then $\kappa_{L/k}^- = \kappa_{L/k} \cap \Cl^-(k)$ is isomorphic 
to a subgroup of ${}_NW_L/W_L^{1-\sigma}$. In particular,
$\# \, \kappa_{L/k}^- \mid p$. 
\end{prop}

Let $k$ be a number field containing a $p$-th root of 
unity $\zeta_p$. A cyclic extension $L/k$ of degree $p$ 
is called {\em essentially ramified} if 
$L = k(\srt{p}{\alpha})$ and $\alpha\OO_k$ is not the 
$p^{th}$ power of an ideal. In particular, subextensions 
of $k(\srt{p}{E_k})/k$ are not essentially ramified.
In \cite{K}, {\sc Kida} showed (generalizing results of 
{\sc Moriya} \cite{Mor} and {\sc Greenberg} \cite{Green}) 
that $\kappa_{L/k}^- = 1$ if $\zeta_p \not\in k$, if 
$L = k(\zeta_{p^n})$ for some $n \ge 1$, or if $L/k$ is 
ramified outside $p$. We will now show that
this result is almost best possible:

\begin{thm}\label{Cap}
Let $p$ be an odd prime and $L/k$ a cyclic $p$-extension of 
CM-fields. Then $\kappa_{L/k}^- = 1$ if and only if
one of the following conditions holds:
\begin{enumerate}
\item[i)] $\zeta_p \not\in k$;
\item[ii)] $\zeta_p \in k$ and $L = k(\zeta_{p^n})$ for some $n \ge 2$;
\item[iii)]  $\zeta_p \in k$ and $L/k$ is essentially ramified.
\end{enumerate}
Moreover, if $\# \, \kappa_{L/k}^- = p$, then 
$L = k(\srt{p}{\beta})$, $\beta \OO_k = \fb^p$, 
and $ \kappa_{L/k}^- = \la [\fb] \ra$. 
\end{thm}

\begin{proof}
Since $G$ is killed by $p$, so is $\text{H}^{-1}(G,E_L) \simeq
{}_NW_L/{}_NW_L \cap E_L^{1-\sigma}$; thus ${}_NW_L = 1$ or
${}_NW_L = \la \zeta_p \ra$. We start by showing $\kappa_{L/k}^- = 1$
if one of the conditions i) -- iii) is satisfied:
\begin{enumerate}
\item[i)] In this case, clearly ${}_NW_L = 1$;
\item[ii)] We have $\zeta_p = \zeta_{p^n}^{1-\sigma}$ for a
	suitable choice of $\sigma$, hence 
	${}_NW_L \subseteq  E_L^{1-\sigma}$.
\item[iii)] Assume that $\kappa_{L/k}^- \ne 1$; then it has order
	$p$, and there exists an ideal class $c = [\fa] \in \Cl_p^-(k)$
	such that $\fa\OO_L = \alpha$ and $\alpha^{1-\sigma} = \zeta_p$. 
	This implies $(\alpha^p)^{1-\sigma} = 1$, i.e.
	$\beta = \alpha^p \in k$, and thus $K = k(\srt{p}{\beta})$.
        But now $\beta\OO_k = \fa^p$, and $K/k$ is not essentially 
	ramified.
\end{enumerate}
Now assume that  $\kappa_{L/k}^- = 1$; if  i) holds,
then we are done, hence we may assume that $\zeta_p \in k$.
Since $K/k$ is cyclic, there exists a $\beta \in k$ such that
$K = k(\srt{p}{\beta})$. If iii) holds, i.e. if $K/k$ is 
essentially ramified, we are done; assume therefore that
$\beta \OO_k = \fb^p$ for some integral ideal $\fb$. Since
$L/F$ is normal, we must have $\beta^{1+J} = \xi^p$ for some
$\xi \in k^\times$ by Galois theory; thus
$(\beta)^{1+J} = (\fb^p)^{1+J} = (\xi)^p$, and we get
$\fb^{1+J} = (\xi)$, in other words: $[\fb] \in \Cl^-(k)$.
But $\fb \OO_L = k(\srt{p}{\beta})$ shows that $\fb$
capitulates, and now our assumption $\kappa_{L/k}^- = 1$ implies
that $\fb \OO_k = (\alpha)$ is principal. Thus
$\beta = \alpha^p \eps$ for some unit $\eps \in E_k$.
Since $\eps^2 = \zeta \eta$ for some root of unity $\zeta \in W_k$
and a real unit $\eta \in E_{k^+}$, we find
$L = k(\srt{p}{\beta}) = k(\srt{p}{\alpha^p \eps})
   =  k(\srt{p}{\eps}) =  k(\srt{p}{\eps^2})
   = k(\srt{p}{\zeta \eta})$. But now
$\beta^{1+J} = \xi^p$ implies that $\eta^2 = (\zeta \eta)^{1+J}$
is also a $p^{th}$ power in $k^\times$, and we finally find
$L = k(\srt{p}{\zeta})$, i.e. we are in case ii).
The last remark follows from the second half of our proof.
\end{proof}

\section{Blowing up class groups}\label{S9}

In this section we will study the behaviour of ideal class
groups under transfer in cyclic extensions. The starting point
of our considerations was the following observation: let
$k$ be a subfield of $K = \Q(\zeta_n)$, and assume that a prime
$p \nmid (K:k)$ divides the class number $h(k)$ of $k$; then
the transfer of ideal classes $j: \Cl(k) \lra \Cl(K)$ is
injective, and $p \mid h(K)$. 

Take for example $n = 23$ and $k = \Q(\sqrt{-23}\,)$: here
$3 \mid h(k)$, $(K:k) = 11$, and hence $3 \mid h(K)$.
This simple method does, however, not explain why $3 \mid h(K)$
for $K = \Q(\zeta_{31})$: although $k = \Q(\sqrt{-31}\,)$ has
class number $3$, the degree $(K:k) = 15$ is divisible by $3$.
From Theorem \ref{Cap} we know that $j: \Cl(k) \lra \Cl(K)$ is
injective in this case also (since $\Cl(k) = \Cl^-(k)$), hence 
$3 \mid h(K)$. The class number formula, on the other hand, 
shows that even $3^2 \mid h(K)$. This is explained by
the following proposition, which generalizes
a result in \cite{Mor} (Satz 2 and \S 4):

\begin{prop}\label{P2.1}
Let $K/k$ be a ramified cyclic extension of prime degree $p$,
put $r = \rank \Cl_p(k)$, and let $\gamma$ denote the rank
of $\kappa_{K/k}$. Then
$ \# \, \Cl_p(K) \ \ge \  p^{r-\gamma}\, \# \, \Cl_p(k).$	
\end{prop}

\begin{proof}
Consider the exact sequence
\begin{equation}\label{CD0}
\begin{CD} 1 @>>> \NCl_p(K) @>>> \Cl_p(K) @>N>> \Cl_p(k) @>>> 1. \end{CD}
\end{equation}
Here the norm map $N: \Cl_p(K) \lra \Cl_p(k)$ is onto by class field
theory since $K/k$ is ramified, and $\NCl_p(K)$ is the kernel of this
map by definition. Now clearly 
$\kappa \subseteq \pCl(k) := \{c \in \Cl(k): c^p = 1\}$ and 
$\pCl(k)^j \subseteq \NCl_p(K)$; this shows immediately that
$\# \, \NCl_p(K) \ge \# \, \pCl(k)^j \ge (\pCl(k):\kappa) = p^{r-\gamma}$. 
\end{proof}

\noindent{\bf Example.} Put $k = \Q(\sqrt{229}\,)$, and let $K$
be the sextic subfield of $\Q(\zeta_{229})$. Computations 
(\cite{Mak}) show that $h(k) = h(K) = 3$: now Prop. \ref{P2.1}
says that the class group of $k$ capitulates in $K$. 

\smallskip

For our next result, we will need some results of {\sc Inaba}
\cite{Ina} (see {\sc Gras} \cite{Gras} for a modern exposition) 
on Galois modules of cyclic groups. Let 
$G = \la \sigma \ra$ be a finite group of prime order $p$, 
and let $M$ be a finite $G$-module of order $p^t$ for some 
$t \in \N$.  Define the submodules 
$M_k = \{ m \in M: m^{(1-\sigma)^k} = 1\}$
and $M^{(k)} = \{ m \in M: m^{p^k} = 1\}$, and let 
$\nu = 1 + \sigma + \sigma^2 + \ldots + \sigma^{p-1}= j \circ N$ 
denote the 'algebraic norm' on $M$. Then
\begin{enumerate}
\item $1 = M_0 \subseteq M_1 \subseteq \ldots \subseteq M_n = M$
	for a sufficiently big $n \in \N$. Moreover, $M_j = M_{j+1}$
	if and only if $M_j = M$;
\item $(M_n:M_{n-1}) \le \ldots \le (M_2:M_1) \le \# \, M_1$;
\item If $M^\nu = 1$, then $M^{(n)} = M_{n(p-1)}$, and in particular
	$\#\, M_1 \le (M:M^p) \le (\# \, M_1)^{p-1}$; if 
	moreover $M^p \ne 1$, then  $(M:M^p)  \ge p^{p-2} \#\, M_1$.
\end{enumerate}
We will also need the existence of polynomials $f, g, h \in \Z[X]$
such that
\begin{eqnarray}
    p   & = & (1-\sigma)^{p-1} f(\sigma) + \nu g(\sigma) \label{E1}\\
   \nu  & = & (1-\sigma)^{p-1} + p h(\sigma) \label{E2} 
\end{eqnarray}

Let $K/k$ be a cyclic extension of prime degree $p$, and
let $\sigma$ be a generator of the Galois group $G = \Gal(K/k)$.
An ideal class $c$ of $\Cl(K)$ is called ambiguous if
it is fixed under the action of $G$, i.e. if $c^\sigma = c$.
The ambiguous ideal classes form a subgroup $\Am(K/k)$ of
$\Cl(K)$.

\begin{prop}\label{P1Am}
If $K/k$ is a cyclic ramified extension of prime degree, then 
$$\#\, \Am(K/k) = (\NCl(K):\Cl(K)^{1-\sigma}) \# \Cl(k);$$
in particular, $\#\, \Am(K/k)$ is divisible by $h_k$.
\end{prop}

\begin{proof}
Applying the snake lemma to the exact and commutative diagram
(note that the surjectivity of the norm map $N:\Cl(K) \lra \Cl(k)$ follows
from class field theory since $K/k$ is completely ramified)
$$\begin{CD}
1 @>>> \Cl(K)^{1-\sigma} @>>> \Cl(K)^{1-\sigma} @>>> 1 \\
  @.    @VVV   @VVV    @VVV   @. \\ 
1 @>>> \NCl(K) @>>> \Cl(K) @>N>> \Cl(k) @>>> 1 \\
\end{CD}$$
and using the fact that the alternating product of the orders
of finite groups in an exact sequence is $1$, we find 
$$ (\NCl(K):\Cl(K)^{1-\sigma}) = (\Cl(K):\Cl(K)^{1-\sigma}) \# \Cl(k). $$
Since $(\Cl(K):\Cl(K)^{1-\sigma}) = \# \Am(K/k)$, 
the claimed equality follows.
\end{proof}

In the special case $\gamma = 0$, Theorem \ref{P3} below was
given (without proof) by Tateyama \cite{T}:

\begin{thm}\label{P3}
If $(\Cl_p(K):\Cl_p(k)^j) = p^a$ for some 
$a \le p-2+\gamma$, then $\Cl_p(k)^j = \Cl_p(K)^p$. 
\end{thm}

\begin{proof}
Put $M = \Cl_p(K)$. We claim that $M^{(1-\sigma)^{p-1}} = 1$.
In fact, assume that this is false. Then $M_{p-1} \ne M$, hence
$p \le (M_p:M_{p-1}) \le \ldots \le (M_2:M_1)$; this shows
$\#\, M \ge (M:M_{p-1})\cdots (M_2:M_1) \#\,M_1 \ge p^{p-1}\#\,M_1$.
Since $M_1 = \pAm(K/k)$, we get $\#\,  \Cl_p(K) \ge p^{p-1}
\#\,  \Cl_p(k) = p^{p-1+\gamma}\#\,  \Cl_p(k)^j$,
where we have used Prop. \ref{P1Am}.

Thus $(1-\sigma)^{p-1}$ kills $M$, and \eqref{E1} and \eqref{E2} 
imply that $M^p = M^\nu$. Since $M^\nu = \Cl_p(k)^j$, the
claim follows. 
\end{proof}

We now define the subgroup $\nCl_p(K)$ by the exact sequence
\begin{equation}\label{CD1}
\begin{CD} 1 @>>> \nCl_p(K) @>>> \Cl_p(K) 
			@>\nu>> \Cl_p(k)^j @>>> 1. \end{CD}
\end{equation}
In other words: $\nCl_p(K)$ is the subgroup of $\Cl_p(K)$
killed by the algebraic norm $\nu$.

\begin{prop} \label{PN}
If $\rank \Cl_p(k)^j \ge \rank \Cl_p(K) - (p-3)$, then
$\nCl_p(K)$ is elementary abelian.
\end{prop}

\begin{proof}
Put $M = \nCl_p(K)$; then $\Cl_p(k)^j \subseteq M_1$, since
$\Cl(k)^j$ is clearly killed by $1-\sigma$. If $M^p \ne 1$, then
$(M:M^p) \ge p^{p-2} \# M_1 $ shows that $\rank \Cl_p(K)
\ge \rank M \ge p-2 + \rank M_1 \ge  p-2 + \rank \Cl(k)^j$,
which contradicts our assumption.
\end{proof}

\begin{thm}\label{TBU}
If $p \ge 3$ and $\rank \Cl_p(k)^j = \rank \Cl_p(K)$, 
then $\Cl_p(k)^j = \Cl_p(K)^p$.
\end{thm}

\begin{proof}
Since $\pCl(k)^j \subseteq \nCl_p(K) \subseteq \Cl_p(K)$, our 
second assumption implies that $\rank \nCl_p(K) = r$.
Thus all groups in the exact sequence (\ref{CD1})
have the same $p$-rank, and now our claim follows since
$\nCl_p(K)$ is elementary abelian by Prop. \ref{PN}.
\end{proof}

For cyclic extensions $K/\Q$, already Moriya \cite{Mor} noticed 
that $\Cl_p(K)$ cannot be cyclic if $\#\, \Cl_p(K) \ge p^2$. 
This was generalized by {\sc Guerry} (\cite{Gue}, Theorem I.9):

\begin{cor} 
If $K/k$ is a cyclic $p$-extension ($p \ge 3$), then 
$\Cl_p(K)$ is cyclic and nontrivial if and only if
$\Cl_p(K)/\Cl_p(k)^j \simeq \Z/p\Z$. 
\end{cor}

\begin{proof}
Assume that $\Cl_p(K)$ is cyclic. If $\Cl_p(k)^j \ne 1$,
then $\Cl_p(K)$ and $\Cl_p(k)^j$ have the same rank (i.e. $1$),
and Theorem \ref{TBU} proves our claim. Assume therefore
that $\Cl_p(k)^j = 1$. Then we have $\Cl_p(K) \simeq \Z/p\Z$
by {\sc Inaba}'s results: put $M = \Cl_p(K)$ and observe 
that $M^\nu = 1$; if $M^p$ were non-trivial, then 
$(M:M^p) \ge p^{p-2} \# M_1 \ge p^2$ (since $M \ne 1$ 
implies $M_1 \ne 1$) shows that $M$ would have rank at 
least $2$, contradicting our assumption. Thus $M^p = 1$, 
and our claim follows.

For the other direction, assume that $\Cl_p(K)/\Cl_p(k)^j \simeq \Z/p\Z$.
Then $a = 1$ in Theorem \ref{P3}, so $\Cl_p(k)^j = \Cl_p(K)^p$
and $(\Cl_p(K):\Cl_p(K)^p) = (\Cl_p(K):\Cl_p(k)^j) = p$, and
$\Cl_p(K)$ is cyclic and non-trivial.
\end{proof}

\begin{rem}\label{R1}
For odd primes $p$, all the results in this section hold with 
$\Cl(K)$, $\Cl_p(K)$ etc. replaced by the corresponding minus 
class groups $\Cl^-(K)$, $\Cl_p^-(K)$ etc. This follows at once 
from the following proposition, which shows that there is
an exact sequence for the minus part of class groups 
corresponding to (\ref{CD0}).
\end{rem}

\begin{prop}
Let $K/k$ be a cyclic extension of CM-fields which is 
completely ramified. Then the following sequence is exact:
$$ \begin{CD}  
1 @>>> \NCl_p^-(K)  @>>> \Cl_p^-(K)  @>N>> \Cl_p^-(k)  @>>> 1.
\end{CD}$$
\end{prop}

\begin{proof}
We only have to show that the norm $N:\Cl_p^-(K) \lra \Cl_p^-(k)$
is surjective. To this end, take a class $c \in \Cl_p^-(k)$, and let
$J$ denote complex conjugation; then
there is an ideal class $C' \in \Cl_p(K)$ such that $c = NC'$. 
But $c \in \Cl_p^-(k)$ implies $c = c^{(1-J)/2)}$ since
$J$ acts as $-1$, and we get
$c = c^{(1-J)/2)} = NC$ for $C = {C'}^{(1-J)/2)}$.
Moreover, $C \in \Cl_p^-(K)$ since it is killed by $1+J$. 
\end{proof}

\noindent{\bf Example.} 
Take $k = \Q(\sqrt{-31}\,)$, and let $K$ be
the sextic subfield of $\Q(\zeta_{31})$. Then 
$\Cl^-(k) \simeq \Z/3\Z$, and the  first condition of 
Theorem \ref{Cap} shows that no class of $\Cl^-(k)$ 
capitulates in $K$. By Prop. \ref{P2.1} we have 
$h^-(K) \equiv 0 \bmod 9$; the tables in \cite{Wa} 
actually show that $h^-(K) = 9$, hence Theorem \ref{P3} 
(applied to the minus part) shows that $\Cl^-(K) \simeq \Z/9\Z$.

\end{document}